\documentclass[preprint,12pt,3p]{elsarticle}

\usepackage{times}
\usepackage{color}
\usepackage[colorlinks=true]{hyperref}
\usepackage{setspace}

\usepackage{amsmath,amssymb,amsthm,amscd}
\usepackage{mathrsfs}
\numberwithin{equation}{section}

\newtheorem{theorem}{Theorem}[section]

\newtheorem{proposition}{Proposition}[section]
\newtheorem{lemma}{Lemma}[section]

\biboptions{sort&compress}

\journal{Elsevier}

\begin{document}
\begin{frontmatter}

\title{Common substring with shifts in $b$-ary expansions}

\author[label1]{Xin Liao}
\ead{xin_liao@whu.edu.cn}
\author[label1]{Dingding Yu\corref{cor1}}
\ead{yudding_sgr@whu.edu.cn}

\address[label1]{School of Mathematics and Statistics, Wuhan University, Wuhan 430072, China}

\cortext[cor1]{corresponding author}

\begin{abstract}
Denote by $S_n(x,y)$ the length of the longest common substring of $x$ and $y$ with shifts in their first $n$ digits of $b$-ary expansions. We show that the sets of pairs $(x,y)$, for which the growth rate of $S_n(x,y)$ is $\alpha \log n$ with $0\le \alpha \le \infty$, have full Hausdorff dimension.

\end{abstract}
\end{frontmatter}
\section{Introduction}
 Fix a positive integer $b\ge 2$. Every $x\in (0,1]$ admits a unique non-terminating $b$-ary expansion:
 $$x=\sum_{i=1}^\infty \frac{x_i}{ b^{i}}:=(x_1,x_2,\cdots),$$
 where $x_i\in \mathcal{A}:=\{0,1,\cdots,b-1\}$ for each $i\ge 1$. The infinite sequence $(x_1, x_2, \cdots)\in \mathcal{A}^{\mathbb{N}}$ is  called the digit sequence of $x$.
 
 For $(x,y)\in (0,1]\times (0,1]$ and $n\in \mathbb{N}$, the length of the longest common substring $L_n(x,y)$, and the length of the longest common substring with shifts $S_n(x,y)$ of $x$ and $y$ in the first $n$ digits of $b$-ary expansions are defined as
\begin{align*}
  L_n(x,y)= & \max \{l\ge 1: x_{i+1}^{i+l}=y_{i+1}^{i+l}~~\text{for some}~ 0\le i \le n-l\},\\
  S_n(x,y)= & \max \{l\ge 1: x_{i+1}^{i+l}=y_{j+1}^{j+l}~~\text{for some}~ 0\le i,j \le n-l\},
\end{align*}
where $x_{i+1}^{i+l}$ denotes the substring $(x_{i+1},\cdots, x_{i+l})$.

It is worth mentioning that the longest common substring problem concerns the shortest distance between two orbits of a  dynamical system $T:X\to X$ over a metric space $(X,d)$. The shortest distance between two orbits is defined as 
$$d_n(x,y)=\min_{i,j=0,\cdots,n-1}(d(T^ix,T^jy)).$$
Barros, Liao, and Rousseau \cite[Section 3]{Liao} pointed out that $S_n(x,y)\le -\log d_n(x,y) \le S_{2n}(x,y)$ for almost all pairs $(x,y)$. Therefore, $S_n(x,y)$ and $-\log d_n(x,y)$ have the same asymptotic behavior. Moreover, It is showed \cite[Theorem 7]{Liao} that if the R\'{e}nyi entropy $H_2$ exists, then for Lebesgue almost all $(x,y)\in (0,1]\times(0,1]$,
$$\lim\limits_{n\to \infty}\frac{S_n(x,y)}{\log n}=\frac{2}{H_2}.$$
Li and Yang \cite{li2019longest} proved that for Lebesgue almost all $(x,y)\in (0,1]\times(0,1]$,
$$
  \lim_{n \to \infty}\frac{L_n(x,y)}{\log_bn}=1.
$$

Then, it is natural to study the points for which $L_n(x,y)$ and $S_n(x,y)$ increase with various speeds.
For $0\le  \alpha \le \infty$, we define the level sets
\begin{equation*}
  L(\alpha)=\Big\{(x,y)\in (0,1] \times (0,1]:\lim_{n\to \infty}\frac{L_n(x,y)}{\log_{b}n}=\alpha\Big\}.
\end{equation*}
In \cite{li2019longest}, Li and Yang proved that all these level sets have full Hausdorff dimension.

However, despite a number of contributions dealing with the sets linked with $L_n(x,y)$, there is few conclusion taking $S_n(x,y)$ into account.  To fill this gap, we aim to extend Li and Yang's results to the level sets associated with $S_n(x,y)$. Different from Li and Yang's method, we mainly exploit the estimation of the spectral radius of the matrix, see Lemma \ref{spectral radius}.

For $0\le \alpha \le \infty$, we define
\begin{equation}\label{S alpha}
  S(\alpha)=\Big\{(x,y)\in (0,1] \times (0,1]:\lim_{n\to \infty}\frac{S_n(x,y)}{\log_{b}n}=\alpha\Big\}.
\end{equation}
The following result shows that all the level sets associated with $b$-ary expansions have full Hausdorff dimension.
\begin{theorem}\label{main theorem1}
  Let $S(\alpha)$ be defined as in (\ref{S alpha}). Then
  \begin{equation*}
    \dim_{\mathcal{H}}S(\alpha)=2,\qquad \forall \alpha\in [0,\infty].
  \end{equation*}
\end{theorem}

Denote
$$S(0,\infty)=\Big\{(x,y)\in (0,1] \times (0,1]:\liminf_{n\to \infty}\frac{S_n(x,y)}{\log_{b}n}=0,~\limsup_{n\to \infty}\frac{S_n(x,y)}{\log_{b}n}=\infty\Big\}.$$

We also investigate the ``sizes" of $S(\alpha)$ and $S(0,\infty)$ from a topological point of view.
\begin{theorem}\label{main theorem3}
 For $0< \alpha< \infty$, the set $S(\alpha)$ is of the first category. The set $S(0,\infty)$ is residual.
\end{theorem}

The paper is organized as follows. In the next section, we compile several lemmas crucial for subsequent proofs of the main theorems. In Sections 3 and 4, we give the proofs of Theorem \ref{main theorem1} and Theorem \ref{main theorem3} respectively.

\section{Preliminaries}

To let our exposition be self-contained, before giving the proof of Theorem \ref{main theorem1}, we present some useful lemmas.

Let $\mathbb{M} $ be a subset of $\mathbb{N}$ and let $\#$ denote the cardinality of a set. We say that the set $\mathbb{M} $ is of density 0 if
$$\lim_{n\to \infty}\frac{\#\{i\in \mathbb{M}: i\le n\}}{n}=0.$$
Write $\mathbb{N} \backslash \mathbb{M}=\{n_1<n_2<\cdots\}$ and define a self-mapping $\varphi_{\mathbb{M}}$ on $(0,1]$ by
$$0.x_1x_2\cdots \mapsto 0.x_{n_1}x_{n_2}\cdots.$$
Let $J\subset (0,1]$. The following lemma describes the relation between the Hausdorff dimensions of $J$ and $\varphi_{\mathbb{M}}(J)=\{\varphi_{\mathbb{M}}(x):x\in J\}$.
\begin{lemma}\label{denisity}(\cite[Lemma 2.3]{chen2014waiting})
Suppose that the set $\mathbb{M}$ is of density zero in $\mathbb{N}$. Then, for any set $J\subset (0,1]$, we have $\dim_{\mathcal{H}}J=\dim_{\mathcal{H}}\varphi_{\mathbb{M}}(J).$
\end{lemma}

For an $n\times n$ matrix $A$ of 0's and 1's having a specified number $\tau$ of 0's, the next lemma gives a lower bound of the spectral radius of $A$ when $\tau \le \lfloor n/2\rfloor\lceil n/2\rceil$, where $\lfloor n\rfloor$ denotes the largest integer no larger than $n$, and $\lceil n\rceil$ denotes the smallest integer no smaller than $n$. Denote by $\rho(A)=\lim_{n\to \infty}\|A^n\|^{\frac{1}{n}}$  the spectral radius of the matrix $A$.

\begin{lemma}(\cite[Theorem 2.1]{brualdi1987minimum})\label{spectral radius}
Let $n$ be a positive integer and let $\tau$ be an integer with $0\le \tau\le \lfloor n/2\rfloor\lceil n/2\rceil$. Denote by $C(n,\tau)$ the class of all $n \times n$ matrices of \ 0's and 1's with exactly $\tau$ \ 0's. Then
  \begin{equation*}
   \rho(A)\geq\frac{1}{2}(n+\sqrt{n^2-4\tau}),\qquad \forall A\in C(n,\tau).
  \end{equation*}
  
\end{lemma}

The following lemma giving a refinement of the Mass Distribution Principle, is usually called Billingsley's lemma. For $x\in (0,1]$, let $I_m(x)$ denote the $m$-th generation, half-open $b$-adic interval of the form $[\frac{j-1}{b^m},\frac{j}{b^m})$ containing $x$ and $|I_m(x)|$ denote the length of the interval $I_m(x)$.
\begin{lemma}\label{Billingsley's Lemma}(\cite[Lemma 1.4.1]{probability}) Let $A \subset (0,1]$ be a Borel set and $\mu$ be a finite Borel measure on $(0,1]$. Suppose $\mu(A)>0$. If
\begin{equation*}
  \beta_1 \le \liminf_{m\to\infty}\frac{\log \mu(I_m(x))}{\log |I_m(x)|} \le \beta_2,
\end{equation*}
 for all $x\in A$, then $\beta_1 \le \dim_{\mathcal{H}}(A)\le \beta_2$.
\end{lemma}

\section{Proof of Theorem \ref{main theorem1}}
Theorem \ref{main theorem1} follows from a series of propositions.

Let $p \geq 2$ be a positive integer. Define
\begin{equation}\label{Ep}
  E_p:=\{x\in (0,1] : x_i=0, ~\text{if} ~
  i=kp^2+j ~\text{for some  } k\in \mathbb{N}~\text{and}~ j\in [1, p+1]
  \}.
\end{equation}
We obtain the Hausdorff dimension of $E_p$ by  the above Billingsley's lemma immediately.
\begin{proposition}\label{Lemma E_p}
Let $E_p$ be defined as in (\ref{Ep}). Then for any $p\ge 2$,
 $$ \dim_{\mathcal{H}}(E_p) =\frac{p^2-(p+1)}{p^2}.$$
\end{proposition}
\begin{proof}
Denote $J=\{kp^2+j: k\in \mathbb{N}~\text{and}~ j\in [p+2, p^2]\}$. Then $E_p$ can be expressed as
\begin{equation}\label{1.4.5}
  E_p=\Big\{x\in(0,1]: x=\sum_{i\in J}\frac{x_i}{ b^{i}},~x_i\in  \mathcal{A} \Big\}.
\end{equation}
The set $E_p$ is covered by exactly $b^{\#(J\cap\{1,\cdots,m\})}$ many closed $b$-adic intervals of generation $m$. Let $\mu$ be the probability measure on $E_p$ that assigns equal measure to the $m$-th generation covering intervals. That is, for any $m\ge 1$, we have
$$\mu(I_m(x))=b^{-\#(J\cap\{1,\cdots,m\})}.$$ 
 This measure makes the digits $\{x_i\}_{i\in J}$ in (\ref{1.4.5}) to be independent and identically distributed uniform random bits. For any $x\in E_p$,
$$\frac{\log \mu(I_m(x))}{\log |I_m(x)|}=\frac{\log b^{-\#(J\cap\{1,\cdots,m\})}}{\log b^{-m}}=\frac{\#(J\cap\{1,\cdots,m\})}{m}.$$
Obviously, $\lim\limits_{m\to \infty}\frac{\#(J\cap\{1,\cdots,m\})}{m}=\frac{p^2-(p+1)}{p^2}.$ Then by Billingsley's lemma \ref{Billingsley's Lemma}, we obtain
$$\dim_{\mathcal{H}}E_p=\lim_{m\to \infty} \frac{\log \mu(I_m(x))}{\log |I_m(x)|}=\frac{p^2-(p+1)}{p^2}.$$
\end{proof}

Define
\begin{equation}\label{Fp}
 F_p:= \{x\in (0,1] :  x_{i}^{i+p}\neq 0^{p+1}, \forall i\ge 1 \},
\end{equation}
where $0^{p+1}$ denotes the substring $(0,0,\cdots,0)$ with $p+1$ zeros.
The following lemma gives the Hausdorff dimension of $F_p$.
\begin{proposition}\label{Lemma F_p}
Let $F_p$ be defined as in (\ref{Fp}). Then for any $p\ge 2$,
\begin{equation*}
    \dim_{\mathcal{H}}F_p=\frac{\log \rho(A)}{\log b^p},
  \end{equation*}
  where $A$ is a matrix of size $b^p \times b^p $, with $pb^{p-1}-(p-1)b^{p-2}$ zeros and all other entries equal to one.
\end{proposition}
\begin{proof}
 For any $i\ge 1$ and $x_i^{2p+i-1}\in \mathcal{A}^{2p}$, there exist non-negative integers $0\leq x'_i, x'_{i+1} \leq b^p-1$ such that $\sum\limits_{j=i}^{2p+i-1}\frac{x_j}{b^j}=\frac{x'_i}{b^{ip}}+\frac{x'_{i+1}}{b^{2ip}}$, which means every $2p$ digits in $b$-ary expansion can be expressed as 2 digits in $b^p$-ary expansion.  Without loss of generality, we can assume $i=1$.

  \begin{itemize}
     \item[(1)]  If $x_1^{p+1}=0^{p+1}$, then $x'_1=0$ and $x'_2\in\{0,\cdots b^{p-1}-1\}$. Thus $(x'_1, x'_2)$ has $b^{p-1}$ values.
     \item[(2)] If $x_2^{p+2}=0^{p+1}$, then $x'_1\in \{0,\cdots b-1\}$ and $x'_2\in\{0,\cdots b^{p-2}-1\}$, when $x'_1=0$, the values of $(x'_1, x'_2)$ have already been counted in case (1), so there are $(b-1)b^{p-2}$ values for $(x'_1, x'_2)$, which are different from case (1).\par
         \dots
     \item[(3)] For any $2\le k < p$, if $x_{k+1}^{p+k+1}=0^{p+1}$, then $x'_1\in \{0,\cdots b^k-1\}$ and $x'_2\in\{0,\cdots b^{p-k-1}-1\}$. When $x'_1\in \{0,\cdots b^{k-1}-1\}$, the values of $(x'_1, x'_2)$ have been appeared in the previous cases, so there are $(b^k-b^{k-1})b^{p-k-1}=(b-1)b^{p-2}$ distinct values for $(x'_1, x'_2)$.
   \end{itemize}
   If there exists $0^{p+1}$ in $x_1\cdots x_{2p}$, then from the preceding discussion, we deduce that $(x'_1, x'_2)$ can take $pb^{p-1}-(p-1)b^{p-2}$ distinct values. Denote the set of these $(x'_1, x'_2)$ by $\Lambda((x'_1, x'_2))$. Then 
   $$\#\Lambda((x'_1, x'_2))=pb^{p-1}-(p-1)b^{p-2}.$$ 
   
   Define a $b^p \times b^p$ matrix $A=(A_{ij}),\ 0\le i,j\le b^p-1$ as follows:
  \begin{equation*}
  A_{ij}=
    \begin{cases}
      0, & \mbox{if } (i,j)\in \Lambda((x'_1, x'_2)); \\
      1, & \mbox{otherwise}.
    \end{cases}
  \end{equation*}
  Then, the number of 0's in matrix $A$ is $pb^{p-1}-(p-1)b^{p-2}$.
Hence, $F_p$ can be expressed as
\begin{equation*}
  F_p=\Big\{x\in (0,1]: x=\sum_{i=1}^{\infty}\frac{x'_i(x)}{b^{ip}}, ~ A_{x'_{i}x'_{i+1}}=1,\ \mbox{for all}~i\Big\}.
\end{equation*}
Using the theory of shifts of finite types (see \cite[Example 1.3.3]{probability} for more details), we can calculate the Hausdorff dimension of $F_p$:
  \begin{equation*}
    \dim_{\mathcal{H}}F_p=\frac{\log \rho(A)}{\log b^p}.
  \end{equation*}
\end{proof}

\begin{proof}[Proof of Theorem \ref{main theorem1}]
First, let us focus on the case $0\le \alpha<\infty$.

For $k\ge 1$, let $m_k=p^2\cdot2^k,~ \ell_k=\lfloor\alpha(k+1)\log_b2\rfloor ~\text{and}~t_k=\sum\limits_{i=1}^{k}(m_i+\ell_i).$
Define self-mapping $f_k$ on $ (0,1]$ by
\begin{align*}
  f_1(x)= & \sum_{i=1}^{m_1}\frac{x_i}{b^i}
+\sum_{i=m_1+1}^{t_1}\frac{1}{b^i}+\sum_{i=t_1+1}^{\infty}
\frac{x_{i-\ell_1}}{b^i} \\
   & \ldots \\
  f_k(x)= & \sum_{i=1}^{t_{k-1}}\frac{x_i}{b^i}
+\sum_{i=t_{k-1}+1}^{t_{k-1}+\ell_k}\frac{1}{b^i}+\sum_{i=t_{k-1}+\ell_k+1}^{\infty}
\frac{x_{i-\ell_k}}{b^i}.
\end{align*}
Then, for any $x\in (0,1]$, the limit
\begin{equation}\label{f}
f(x):=\lim_{k\to\infty} f_k(f_{k-1}(\ldots f_1(x))),
\end{equation}
exists.

For each $x\in f(E_p)$ and $y\in f(F_p)$, if $t_k \le n < t_{k+1}$, then
 $$\ell_k \le S_n(x,y)< \ell_{k+1}+2p^2.$$
 It follows that
 $$\lim_{n\to \infty}\frac{S_n(x,y)}{\log_bn}=\alpha,$$
 and hence $$f(E_p)\times f(F_p) \subset S(\alpha).$$
 Therefore, using \cite[Product formula 7.2]{falconer2004fractal}, we have
 \[\dim_{\mathcal{H}}S(\alpha)\geq \dim_{\mathcal{H}}f(E_p)+ \dim_{\mathcal{H}}f(F_p).\]
 On the other hand, since $\lim\limits_{k\to \infty}\frac{\sum_{i=1}^{k}\ell_i}{t_k}= 0$, the assumption of Lemma \ref{denisity} is satisfied for the sets $E_p$ and $F_p$. Hence, we obtain
 $$\dim_{\mathcal{H}} f(E_p)=\dim_{\mathcal{H}} E_p,$$
 $$\dim_{\mathcal{H}} f(F_p)=\dim_{\mathcal{H}} F_p.$$
 Therefore, utilizing Proposition \ref{Lemma E_p} and \ref{Lemma F_p}, we have
 $$\dim_{\mathcal{H}}S(\alpha)\geq \frac{p^2-p-1}{p^2}+ \frac{\log \rho(A)}{\log b^p},$$
 where $A$ is defined as in Proposition \ref{Lemma F_p}. By Lemma \ref{spectral radius}, we have
 \[ \rho(A) \geq  \frac{1}{2}\big(b^p+\sqrt{b^{2p}-4(pb^{p-1}-(p-1)b^{p-2})}\big). \]
 Taking $p\to \infty$ yields
 $$\dim_{\mathcal{H}}S(\alpha)\geq  \lim_{p\to \infty}\frac{p^2-p-1}{p^2}+ \frac{\log \rho(A)}{\log b^p}=2.$$
 Thus, we have shown the desired result for $0\leq \alpha <\infty$.

As for $\alpha=\infty$, the proof needs to be modified accordingly. For this case, we define $m_k=p^2\cdot 2^k$ and $\ell_k=k^2$ for $k\geq 1$. Then, the rest of the proof proceeds as before.

\end{proof}

\section{Proof of Theorem \ref{main theorem3}}

  For $\alpha\in [0,\infty]$, set
  \begin{align*}
    \underline{S}^*(\alpha)= & \Big\{(x,y)\in (0,1]^2:\liminf_{n\to \infty}\frac{\log S_n(x,y)}{\log_b n}\le \alpha\Big\}, \\
    \overline{S}_{*}(\alpha)= & \Big\{(x,y)\in (0,1]^2:\limsup_{n\to \infty}\frac{\log S_n(x,y)}{\log_b n}\ge \alpha\Big\}.
  \end{align*}
  
The proof of Theorem \ref{main theorem3} consists of several claims.

\textbf{Claim 1}:
    For any $0\le \alpha\le\infty$, $S(\alpha)$ is dense in $(0,1]^2$.

  \begin{proof}
    For any $\alpha\in [0,\infty]$, there exists $(x',y')\in (0,1]^2$ such that
    $$\lim_{n\to \infty}\frac{\log S_n(x',y')}{\log_b n}=\alpha.$$
    For any $(x,y)\in (0,1]^2$, we can find a sequence of points in $S(\alpha)$,
    \begin{align*}
      \widetilde{x}_k:= & \sum_{i=1}^{k}\frac{x_i}{b^i} +\sum_{i=k+1}^{\infty}\frac{x'_i}{b^i}, \\
      \widetilde{y}_k:= & \sum_{i=1}^{k}\frac{y_i}{b^i}
      +\sum_{i=k+1}^{\infty}\frac{y'_i}{b^i},
    \end{align*}
    such that $(\widetilde{x}_k,\widetilde{y}_k)\to (x,y)$ as $n\to \infty$. Indeed, for any $k\ge 1$, $\widetilde{x}_k(\widetilde{y}_k,\text{respectively})$ and $x'(y',\text{respectively})$ differ only in finitely many digits. Then
    $$\lim_{n\to \infty}\frac{\log S_n(\widetilde{x}_k,\widetilde{y}_k)}{\log_b n}=\lim_{n\to \infty}\frac{\log S_n(x',y')}{\log_b n}=\alpha.$$
     Thus $S(\alpha)$ is dense in $(0,1]^2$. 
    
  \end{proof}


  \textbf{Claim 2}:
    For any $0<\alpha<\infty$, $\underline{S}^*(\alpha)$ and $\overline{S}^*(\alpha)$ are residual.
  
  \begin{proof}
    Let $\alpha \in (0,\infty)$ be fixed. From Claim 1, we can deduce that $\underline{S}^*(\alpha)$ and $\overline{S}^*(\alpha)$ are dense. Since $(0,1]^2$ is a Baire space, it suffices to show that $\underline{S}^*(\alpha)$ and $\overline{S}^*(\alpha)$ are $G_{\delta}$ sets.
    
    We observe that, for any $k>0$,
    $$\underline{S}^*(\alpha)=\bigcap_{k=1}^{\infty}\bigcap_{N=1}^{\infty}\bigcup_{n=N}^{\infty}B_n(\alpha,k),$$
    and
    $$\overline{S}^*(\alpha)=\bigcap_{k=\lfloor 1/\alpha\rfloor+1}^{\infty}\bigcap_{N=1}^{\infty}\bigcup_{n=N}^{\infty}\hat{B}_n(\alpha,k),$$
    where $B_n(\alpha,k)$ and $\hat{B}_n(\alpha,k)$ are defined by
    $$B_n(\alpha,k):=\{(x,y)\in (0,1]^2:S_n(x,y)< n^{\alpha+1/k}\},$$
    and
    $$\hat{B}_n(\alpha,k):=\{(x,y)\in (0,1]^2:S_n(x,y)> n^{\alpha-1/k}\}.$$
    All non-empty sets $B_n(\alpha,k)$ and $\hat{B}_n(\alpha,k)$ are  open sets, implying that $\underline{S}^*(\alpha)$ and $\overline{S}^*(\alpha)$ are $G_{\delta}$ sets.

\end{proof}

\begin{proof}[Proof of Theorem \ref{main theorem3}]
  For any $K\in \mathbb{N}$, we have
  $$\underline{S}^*(0)=\bigcap_{K=1}^{\infty}\underline{S}^*(1/K)\quad \text{and}\quad \overline{S}^*(\infty)=\bigcap_{K=1}^{\infty}\overline{S}^*(K).$$
   From Claim 2, we conclude that $\underline{S}^*(1/K)$ and $\overline{S}^*(K)$ are residual. Then $\underline{S}^*(0)$ and $\overline{S}^*(\infty)$ are residual. Hence the set
   $$S(0,\infty)=\underline{S}^*(0)\bigcap \overline{S}^*(\infty)$$
   is residual.
   
   For any $0< \alpha <\infty$, we have
   $$S(\alpha)\subset \Big(\underline{S}^*(0)\bigcap \overline{S}^*(\infty)\Big)^{c}.$$
    By the definition of the set of the first category, we have $S(\alpha)$ is of first category.
\end{proof}

\section*{Acknowledgments}
The authors wish to express their sincere appreciation to Professor Lingmin Liao who critically read the paper and made numerous helpful suggestions.

\end{document}